\documentclass[a4paper,11pt]{article}
\date{}
\title{
	Worst-case error for unshifted lattice rules without randomisation
}
\usepackage{pdfsync}

\usepackage{amsmath} % AMS Math Package
\numberwithin{equation}{section}% add section number to label
\usepackage{amsthm} % Theorem Formatting
\usepackage{amssymb}    % Math symbols such as \mathbb
\usepackage{mathrsfs}
\usepackage{comment}
\usepackage[backref=page]{hyperref}
\renewcommand*{\backref}[1]{}
\renewcommand*{\backrefalt}[4]{[{\normalsize%
		\ifcase #1 Not cited.%
		\or Cited on page~#2.%
		\else Cited on pages #2.%
		\fi%
	}]}

\usepackage{graphicx,subfig,here}%YK for graphics

\usepackage[normalem]{ulem} % to use \sout{text}, strikethrough text

\usepackage{accsupp}
\theoremstyle{theorem}
\newtheorem{theorem}{Theorem}[section]
\newtheorem{proposition}[theorem]{Proposition}
\newtheorem{lemma}[theorem]{Lemma}
\newtheorem{corollary}[theorem]{Corollary}
\newtheorem{conjecture}[theorem]{Conjecture} %FK -- make the conjecture runs with the theorems

\theoremstyle{definition}
\usepackage{color}
\definecolor{refkey}{rgb}{0.9451,0.2706,0.4941}\definecolor{labelkey}{rgb}{0.9451,0.2706,0.4941}

\definecolor{darkred}{RGB}{139,0,0}
\definecolor{darkgreen}{RGB}{0,100,0}
\definecolor{darkmagenta}{RGB}{139,0,139}

\newcommand{\setu}{{\mathrm{\mathfrak{u}}}}
\newcommand{\setv}{{\mathrm{\mathfrak{v}}}}

\newcommand{\bsx}{{\boldsymbol{x}}}
\newcommand{\bsy}{{\boldsymbol{y}}}
\newcommand{\bsz}{{\boldsymbol{z}}}
\newcommand{\bst}{{\boldsymbol{t}}}
\newcommand{\bsDelta}{\boldsymbol{\Delta}}
\newcommand{\bsgamma}{{\boldsymbol{\gamma}}}
\newcommand{\rd}{{\mathrm{d}}}
\newcommand{\bbZ}{{\mathbb{Z}}}
\newcommand{\bbN}{{\mathbb{N}}}
\newcommand{\bbR}{{\mathbb{R}}}
\newcommand{\satop}[2]{\stackrel{\scriptstyle{#1}}{\scriptstyle{#2}}}

\author{Yoshihito Kazashi, Frances Y.\ Kuo, and Ian H.\ Sloan}

\date{13 November 2018}

\begin{document}

\maketitle
\begin{abstract}
An existence result is presented for the worst-case error of lattice rules
for high dimensional integration over the unit cube, in an unanchored
weighted space of functions with square-integrable mixed first
derivatives. Existing studies rely on random shifting of the lattice to
simplify the analysis, whereas in this paper neither shifting nor any
other form of randomisation is considered. Given that a certain
number-theoretic conjecture holds, it is shown that there exists an
$N$-point rank-one lattice rule which gives a worst-case error of order
$1/\sqrt{N}$ up to a (dimension-independent) logarithmic factor. Numerical
results suggest that the conjecture is plausible.
\end{abstract}

\section{Introduction}
This paper is concerned with an error estimate for a numerical integration
rule for functions defined on high-dimensional hypercube $[0,1)^{s}$,
$s\in\mathbb{N}$,
\begin{equation} \label{eq:int}
	\int_{[0,1)^{s}}f(\bsx)\,\rd\bsx.
\end{equation}
More specifically, we consider the worst-case error for rank-one lattice
rules. The main contribution of this paper is the analysis of unshifted
lattice rules without randomisation; we allow neither shifting nor any
other form of randomisation. Given the truth of a certain conjecture with
a number-theoretic flavour {(Conjecture~\ref{conj:alpha})}, our results
show the existence of a deterministic cubature point set that attains the
worst-case error of the order $1/\sqrt{N}$,  up to a logarithmic factor,
where $N$ is the number of cubature points, with a dimension-independent
constant (Corollary~\ref{cor:final-cor}). 	

An $N$-point rank-one lattice rule in $s$-dimension is an equal-weight
cubature rule for approximating the integral \eqref{eq:int} --- a
quasi-Monte Carlo rule --- of the form
\begin{equation} \label{eq:qmc}
	\frac1{N}\sum_{k=0}^{N-1}f(\bst_k),
\end{equation}
with cubature points
\begin{equation} \label{eq:lat}
  \bst_{k} =\bigg\{\frac{k\bsz}{N}\bigg\} ,\quad k=0,\dotsc, N-1,
\end{equation}
for some $\bsz \in \{1,\dotsc ,N-1\}^{s}$, where $\{\bsx\}\in [0,1)^s$ for
$\bsx=(x_1,\dots, x_s)\in [0,\infty)^{s}$ denotes the vector consisting of
the fractional part of each component of $\bsx$. The choice of $\bsz$,
known as the \textit{generating vector}, completely determines the
cubature points, and thus the quality of the cubature rule. Our interest
in this paper lies in proving the existence of a good generating vector
$\bsz\in \{1,\dotsc ,N-1\}^{s}$. The figure of merit we consider is the
so-called worst-case error, defined by
\[
	e(N,\bsz):=
	e(N,(\{k\bsz/N\})_k):=
	\sup\limits_{f\in H_{s,\bsgamma},\,\|f\|_{H_{s,\bsgamma}}\le 1}
	\bigg|
	\int_{[0,1)^{s}} f(\bsx)\,\rd\bsx
	-
	\frac1{N}\sum_{k=0}^{N-1}f(\{{k\bsz}/{N}\}) \bigg|,
\]
where $H_{s,\bsgamma}$ is a suitable normed space consisting of
non-periodic functions over $[0,1)^s$, specified below. As is standard
nowadays, we will assume that the norm incorporates certain
parameters~$\gamma_\setu$, one for each subset $\setu\subseteq
\{1,2,\ldots,s\}$, since without weights integration problems are often
intractable{, see {\cite{Dick.J_etal_2013_Acta,SW98}} for more
details}.

It is natural to seek a generating vector $\bsz$ that makes the worst-case
error small. If $H_{s,\bsgamma}$ is a reproducing kernel Hilbert space
then the worst-case error $e(N,\bsz)$ can be computed for any value of
$\bsz$ (see below), but there is no known formula that gives a good value
of $\bsz$ for general $s$. The strategy we take in this paper is to prove
an existence result, by considering the average of $e^2(N,\bsz)$ over all
possible generating vectors $\bsz\in Z_{N}^s$, with $Z_N:=\{1,2,
\ldots,N-1 \}$, i.e. we compute
\begin{equation} \label{eq:avg}
  \overline{e}^{2}(N ):=\frac1{(N-1)^s}\sum_{\bsz\in Z_{N}^s} e^2(N,\bsz);
\end{equation}
and then use the well known principle that there must exist one choice of
$\bsz$ that is as good as average.
With the support of a certain number-theoretic conjecture
({Conjecture~\ref{conj:alpha}}), which does not depend on the choice of
$\bsz$), we will show that {
\[
  \overline{e}^{2}(N ) \le \frac{C\,(\ln N)^{\alpha}}{N},
\]}%
with $C$ independent of $N$, where $\alpha>0$ is an exponent appearing in
the conjecture that depends on neither $s$ nor $N$. Moreover, $C$ is
independent of $s$ for suitable weights $\gamma_\setu$. It follows that
there exists a generating vector $\bsz^*$ for which the worst-case error
$e(N,\bsz^*)$ is bounded by $\sqrt{C}(\ln N)^{\alpha/2}/\sqrt{N}$
(Corollary~\ref{cor:final-cor}).

For periodic function spaces, error estimates for rank-one lattice rules
are well known; see
{\cite{Dick.J_etal_2013_Acta,Hic98b,Nie92,Sloan.I.H_Joe_1994_book}} and
references therein.
For non-periodic functions, with the aid of \textit{shifting}---changing
the cubature points from $\{k\bsz/N\}$ to $\{k\bsz/N + \bsDelta\}$ with
elements $\bsDelta\in [0,1)^s$---good results have been obtained for
shift-averaged worst-case errors; see \cite{Dick.J_etal_2013_Acta} and
references therein for more details.
In the present paper, however, the function space is not periodic, and the
worst-case error we consider is not shift-averaged.	Approaches to
estimating the error for lattice rules for non-periodic functions without
randomisation include \cite{Dick.J_etal_2014_tent,Goda.T_etal_2017_tent}
where a change of variable called the tent transform was applied to the
integrand. In this paper, however, we do not transform the integrand.

The shift-averaged worst-case error mentioned above is the expected
worst-case error for \textit{randomly shifted} lattice rules, see
\cite{Dick.J_etal_2013_Acta}. The present paper is a first step in our
project to ``\textit{derandomise}'' randomly shifted lattice rules---that
is, to produce explicit shifts  (for an untransformed rule) that gives
worst-case errors that lose no accuracy compared to the shift-averaged
worst-case errors. While randomly shifted lattice rules have the
advantages of {providing us with an online error estimator and are} simple to  analyse and construct, they 
are less efficient than a good deterministic rule, because of the need in
practice to repeat the calculations of integrals with fixed $\bsz$ for
some number (say $30$) of random shifts.
In this first step in this programme, we study the case of zero shifts.
(Experience suggests that this is a poor choice---perhaps the worst!)

{There are related works in \cite{Joe04,Joe06} where a quantity called
`$R$', which is connected to the so-called (weighted) star discrepancy,
was considered as the error criterion. In the weighted setting in
\cite{Joe06}, lattice rules can be constructed to achieve
$O(n^{-1+\delta})$ convergence rate for any $\delta
>0$, with the implied constant independent of $s$ and $N$ for suitable weights.}

After establishing the setting in Section~\ref{sec:preliminaries}, the
conjecture and the main results are stated in Section~\ref{sec:estimate}.
Section~\ref{sec:num} provides numerical evidence relating to the
conjecture. Section~\ref{sec:conclusion} concludes the paper. 	

\section{Preliminaries}\label{sec:preliminaries}

In this section, we introduce the setting and recall some facts on lattice
rules that will be needed later. Throughout this paper, we assume that
$N$, the number of cubature points, is a prime number. Let us start with a
general reproducing kernel Hilbert space (RKHS) $H_{s}$ with a reproducing
kernel $K\colon[0,1]^s\times [0,1]^s\to \mathbb{R}$ that satisfies
\[
  \int_{[0,1]^s}\int_{[0,1]^s} K(\bsx,\bsy) \,\rd\bsx\,\rd\bsy < \infty.
\]
It is well known that for a general quasi-Monte Carlo (QMC) rule
\eqref{eq:qmc}, the square of the worst-case error in $H_s$,
\[
  e(N,(\bst_k)_k):=\sup\limits_{f\in H_{s},\,
  \|f\|_{H_{s}}\le 1}
  \bigg|
  \int_{[0,1)^{s}} f(\bsx)\,\rd\bsx
  - \frac1{N}\sum_{k=0}^{N-1}f(\bst_k) \bigg|,
\]
is given by
\begin{align*}
  &e^{2}(N,(\bst_k)_k) \\
  &=
	\int_{[0,1]^{s}}\int_{[0,1]^{s}}
	K(\bsx,\bsy)\,\rd\bsx\,\rd\bsy
	-\frac{2}{N}\sum_{k=0}^{N-1} \int_{[0,1]^{s}} K(\bst_{k} ,\bsx)\,\rd\bsx
	+\frac{1}{N^{2}}\sum_{k=0}^{N-1}\sum_{k'=0}^{N-1} K(\bst_{k} ,\bst_{k'}),
\end{align*}
see for example \cite[Theorem 3.5]{Dick.J_etal_2013_Acta}.
We specialise to the case
\begin{equation*}
  \int_{[0,1]^{s}} K(\bsx,\bsy)\,\rd\bsx =1 \qquad\text{for any } \bsy\in [0,1]^{s} ,
\end{equation*}
to obtain
\begin{align} \label{eq:err1}
  e^{2}(N,(\bst_k)_k)
  &= \frac{1}{N^{2}}\sum_{k=0}^{N-1}\sum_{k'=0}^{N-1} K(\bst_{k},\bst_{k'}) -1.
\end{align}
In particular, for the QMC rule we here take an unshifted lattice rule
with cubature points given by \eqref{eq:lat} for some $\bsz \in Z_N^s$.
Then, we have
\begin{equation}
	e^{2}(N,\bsz) =
	\frac{1}{N^{2}}
	\sum_{k=0}^{N-1}\sum_{k'=0}^{N-1}
	K\left(\bigg\{\frac{k\bsz}{N}\bigg\} ,\bigg\{\frac{k'\bsz}{N}\bigg\}\right) -1.
	\label{eq:decomp-err-in-u}
\end{equation} 	
Now we further specialise the RKHS to $H_{s,\bsgamma}$ with kernel
\begin{equation}
	K_{s,\bsgamma }(\bsx,\bsy) =\sum_{\setu \subseteq \{1:s\}} \gamma_{\setu}\prod_{j\in \setu } \eta ( x_{j} ,y_{j}) ,
	\label{eq:def-kernel}
\end{equation}
where
\begin{equation*}
  \eta(x,y) :=\frac{1}{2} B_{2}( |x-y|) +\Big( x-\frac{1}{2}\Big)\Big( y-\frac{1}{2}\Big) ,\qquad x,y\in [0,1] .
\end{equation*}
Here $B_2(t)=t^2-t+1/6$, $t\in\mathbb{R}$ is the Bernoulli polynomial of
degree $2$, $\{1:s\}$ is a shorthand notation for $\{1,2,...,s\}$, and the
sum in \eqref{eq:def-kernel} is over all subsets $\setu\subseteq\{1:s\}$,
including the empty set; and $\bsgamma = \{\gamma_{\setu}\}_{\setu\subset
\mathbb{N}}$ is an arbitrary collection of positive numbers called
\textit{weights} with $\gamma_{\emptyset}=1$. The choice of weights plays
an important role in deriving a dimension-independent error estimate, see
Corollary~\ref{cor:final-cor}. This space, discussed fully in
\cite{Dick.J_etal_2013_Acta}, is an ``unanchored'' space of functions on
the unit cube with square integrable mixed first derivatives. We again
refer the readers to \cite{Dick.J_etal_2013_Acta} for more details. For
this space it follows from \eqref{eq:err1} that 		
\begin{equation}
	e^{2}(N,\bsz)
	=\sum_{{\emptyset \neq } \setu \subseteq \{1:s\}} \gamma_{\setu} \, e^{2}_{\setu}(N,\bsz_\setu) ,
	\label{eq:def-e2}
\end{equation}
where for $\setu \subseteq \{1:s\}$ and $\bsz_\setu = (z_j)_{j\in\setu}$,
from \eqref{eq:err1} and \eqref{eq:def-kernel}
\begin{align}
	& e^{2}_{\setu}(N,\bsz_\setu)\notag\\
	&
	:=
	\frac{1}{N^{2}}
	\sum_{k=0}^{N-1}
	\sum_{k'=0}^{N-1}\prod_{j\in \setu}
	\bigg[
	\frac{1}{2}
	B_{2}\bigg(
	\bigg|\bigg\{\frac{kz_{j}}{N}\bigg\} -\bigg\{\frac{k'z_{j}}{N}\bigg\} \bigg|
	\bigg)
	+\left(\bigg\{\frac{kz_{j}}{N} \bigg\} -\frac{1}{2}\right)
	\left(\bigg\{\frac{k'z_{j}}{N} \bigg\} -\frac{1}{2}\right)
	\bigg] .
	\label{eq:def-eu2}
\end{align}
Thus the quantity $e^{2}_{\setu}(N,\bsz_\setu)$ is a key to deriving an
estimate for $e^{2}(N,\bsz)$.

\section{Existence result for worst-case error}\label{sec:estimate}

In this section, we derive an existence result for the worst-case error.
We first note the following property.
\begin{proposition}\label{prop:property-B2}
Let $g$ be a function that satisfies $g(t)=g(1-t)$ for $t\in [0,1]$. Then
for $a,b\ge0$ we have
\[
 g(|\{a\} -\{b\}|) = g(\{a-b\}),
\]
where, as before, the braces indicate that we take the fractional part of
the real number.
\end{proposition}

\begin{proof}
Note first that $\{a\},\{b\}\in[0,1)$ and therefore $\{a\} -\{b\}\in
(-1,1)$.  It is clear that  $\{a\} -\{b\}$ differs from $\{a-b\}$ by $1$
or $0$. If $\{a\}=\{b\}$, then $\{a-b\}=0$ and the result is trivial. If
$\{a\}>\{b\}$, then $\{a\}-\{b\}\in (0,1)$, and so $\{a\}-\{b\}=\{a-b\}$.
Thus, again the result is trivial. If $\{a\}<\{b\}$, then
$|\{a\}-\{b\}|=\{b\}-\{a\}\in (0,1)$ and so $|\{a\}-\{b\}|=\{b-a\}$. Thus,
using $g(t)=g(1-t)$, $t\in [0,1]$ we have
\[
  g(|\{a\}-\{b\}|)
  = g(\{b\}-\{a\}) = g(\{b-a\}) = g(1-\{b-a\}) = g(\{a-b\}),
\]
where in the last step we used the identity $\{t\}+\{-t\}=1$ for $t\not\in
\mathbb{Z}$.
\end{proof}
In particular, Proposition~\ref{prop:property-B2} applies to the function
$B_2(\cdot)$ so we can rewrite \eqref{eq:def-eu2} as
\begin{align}
	&e^{2}_{\setu}(N,\bsz_\setu) \notag\\
	&=\frac{1}{N^{2}}\sum_{k=0}^{N-1}\sum_{k'=0}^{N-1}\prod_{j\in \setu }
	\bigg[
	\frac{1}{2}
	B_{2}\left( \bigg\{\frac{( k-k') z_{j}}{N}\bigg\} \right)
	+\left(\bigg\{\frac{kz_{j}}{N} \bigg\} -\frac{1}{2}\right)\left(\bigg\{\frac{k'z_{j}}{N} \bigg\} -\frac{1}{2}\right)
	\bigg].\label{eq:(6)}
\end{align}
Now we obtain the average over $\bsz \in Z_{N}^s$. From \eqref{eq:avg} and
\eqref{eq:def-e2} we have
\begin{equation}
	\overline{e}^{2}(N)
	=\sum_{\emptyset \neq \setu \subseteq \{1:s\}} \gamma_{\setu}\, \overline{e}^{2}_{\setu}(N),
	\label{eq:def-bar-e2N}
\end{equation}
where
\begin{align}
	\overline{e}^{2}_{\setu}(N)
	:=
	\frac{1}{(N-1)^{s}}\sum_{\bsz \in Z_{N}^s}
	e^{2}_{\setu}(N,\bsz_\setu)
	&=
	\frac{1}{(N-1)^{|\setu |}}\sum_{\bsz_\setu \in Z_N^{|\setu|}} e^{2}_{\setu}(N,\bsz_\setu) \nonumber\\
	&
	= \frac{1}{N^{2}}
	\sum_{k=0}^{N-1}
	\sum_{k'=0}^{N-1}
	(
	X_{N;k,k'} + J_{N;k,k'}
	)^{|\setu|} ,\label{eq:def-bar-e-u}
\end{align} 	
{with}
\begin{equation}
	X_{N;k,k'}:=
	\frac{1}{2(N-1)}\sum ^{N-1}_{z =1}
	B_{2}\left( \bigg\{\frac{( k-k') z}{N}\bigg\} \right) ,    \label{eq:def-XNkk'}
\end{equation}
and
\begin{equation}
	J_{N;k,k'} :=
	\frac{1}{N-1}\sum_{z=1}^{N-1}
	\bigg(\bigg\{\frac{kz}{N} \bigg\}-\frac{1}{2}\bigg)
	\bigg(\bigg\{\frac{k'z}{N} \bigg\}-\frac{1}{2}\bigg). \label{eq:def-JNkk'}
\end{equation}
Further, the binomial theorem gives
\begin{equation}
	\overline{e}^{2}_{\setu}(N)
	=
	\frac{1}{N^{2}}
	\sum_{k=0}^{N-1}
	\sum_{k'=0}^{N-1}
	\sum_{ \setv\subseteq \setu }
	(X_{N;k,k'})^{|\setu\setminus \setv|}
	(J_{N;k,k'})^{|\setv|}
	.\label{eq:e2binom}
\end{equation}
In seeking an error estimate for the generating-vector-averaged worst-case
error $\overline{e}^{2}(N)$, we take the point of view that estimates of
order $1/N$ or higher are relatively harmless, so we are concentrating on
isolating terms that are more slowly converging.	

In the following two subsections, we derive estimates for $X_{N;k,k'}$ and
$J_{N;k,k'}$. It turns out that, roughly speaking, the terms
$(X_{N;k,k'})^{|\setu\setminus \setv|}$ yield the order $1/N$. The terms
$(J_{N;k,k'})^{|\setv|}$ seem to converge more slowly, and require more
detailed analysis. 	

\subsection{Estimates for $X_{N;k,k'}$}\label{estimaesfor}

We have the following expression for $X_{N;k,k'}$.
\begin{lemma}\label{lem:estim-XNkk'}
For $N$ prime and $k,k'\in \{0,1,\dots,N-1\}$, the quantity $X_{N;k,k'}$
defined in \eqref{eq:def-XNkk'} satisfies {
\begin{equation}
 X_{N;k,k'}=
 \begin{cases}
 \displaystyle\frac1{12}   & \mbox{if } k=k', \\[5pt]
 -\displaystyle\frac1{12N} & \mbox{if } k\not=k'.
  \end{cases}
\end{equation}}%
\end{lemma}
\begin{proof}
For $k=k'$, we have $B_2(0)=\frac16$. For $k\not=k'$, recalling
 the (absolutely convergent) series representation
\[
  B_{2}(x)=\frac{1}{2\pi^{2}}\sum_{h\ne 0}\frac{\exp(2\pi
  \mathrm{i}hx)}{h^{2}}, \qquad x\in [0,1],
\]
we have
\begin{align*}
  X_{N;k,k'}
  &=\frac{1}{4\pi^{2}(N-1)}\sum_{h\ne 0}\frac{1}{h^{2}}
    \sum_{z=1}^{N-1}\exp( 2\pi \mathrm{i}h( k-k') z/N)\\
  &=\frac{1}{4\pi^{2}(N-1)}\sum_{h\ne 0}\frac{1}{h^{2}}\left(\sum ^{N-1}_{z=0}\exp( 2\pi \mathrm{i}zh( k-k') /N) -1\right) ,
\end{align*}%
with
\begin{equation*}
		\sum ^{N-1}_{z=0}\exp( 2\pi \mathrm{i}zh(k-k')/N) =\begin{cases}
		N & \text{if}\ \ h(k-k')\equiv_N 0, \\
		0 & \text{if}\ \ h(k-k')\not\equiv_N 0.
		\end{cases} %\label{eq:character1}
\end{equation*}
Throughout this paper, the notation $a \equiv_N b$ means that $a\equiv b
\pmod N$, and similarly $a \not\equiv_N b$ means that $a\not\equiv b \pmod
N$. Since $N$ is prime and $k\ne k'$, we conclude that all possible values
of $k-k'$, namely, $\pm 1,\pm 2, \ldots$, $\pm (N-1)$, are relatively
prime to $N$, and so $h(k-k')\equiv_N 0 \iff h\equiv_N 0$. Thus
\begin{align*}
		X_{N;k,k'}
		&=
		\frac{1}{4\pi^{2}(N-1)}
		\Bigg(N\sum_{\substack{
				h\neq 0\\
				h\equiv_N 0
		}}\frac{1}{h^{2}} -\sum_{h\ne 0}\frac{1}{h^{2}}\Bigg)\\
		&=\frac{1}{4\pi^{2}(N-1)}\bigg(N\sum_{\ell \not{=} 0}\frac{1}{(\ell N)^{2}} -\frac{\pi^{2}}{3}\bigg)
        =\frac{1}{4\pi^{2}(N-1)}\left(\frac{N}{N^{2}}\frac{\pi^{2}}{3} -\frac{\pi^{2}}{3}\right)
		=-\frac{1}{12N},
\end{align*}%
which completes the proof.
\end{proof}

We deduce the following estimate for $\overline{e}^{2}_{\setu}(N)$.
	
\begin{proposition}\label{prop:bd-1st-term-ofe2}
For $N$ prime, the quantity $\overline{e}^{2}_{\setu}(N)$ defined in
\eqref{eq:e2binom} satisfies
\begin{equation*}
		\overline{e}^{2}_{\setu}(N)
		\le c_{\setu}
		\frac1N
		+
		\frac1{N^2}
		\sum_{k=1}^{N-1}\sum_{k'=1}^{N-1}
		(J_{N;k,k'})^{|\setu|},
  \qquad \mbox{with}\quad
  c_{\setu}:=\frac{2}{3^{|\setu |}} + \frac{1}{4^{|\setu |}}.
\end{equation*}
\end{proposition}

\begin{proof}
On separating out the diagonal terms of \eqref{eq:e2binom}, we have
\begin{align} \label{eq:two-terms}
	\overline{e}^{2}_{\setu}(N)		
& =\frac{1}{N^{2}}\sum_{k=0}^{N-1}
			(X_{N;k,k} +J_{N;k,k} )^{|\setu|}
			+\frac{1}{N^{2}}\sum_{k=0}^{N-1}\sum\limits ^{N-1}_{\substack{
				k'=0\\
				k'\neq k}}\sum_{\setv \subseteq \setu}
   (X_{N;k,k'})^{|\setu \setminus \setv|} (J_{N;k,k'})^{|\setv|} .
\end{align}
{}From $X_{N;k,k} =\frac{1}{12}$ and $0\le J_{N;k,k}\leq \frac{1}{N-1}\sum
^{N-1}_{z=1}\frac{1}{4} =\frac{1}{4}$, the first term in
\eqref{eq:two-terms} can be bounded by		
\begin{equation*}
		\frac{1}{N^{2}}\sum_{k=0}^{N-1}
		(X_{N;k,k} +J_{N;k,k} )^{|\setu|}\leq \frac{1}{3^{|\setu |} N} .
\end{equation*}
For the second term in \eqref{eq:two-terms}, noting $ |J_{N;k,k'} |\leq
\frac{1}{4}$, from Lemma~\ref{lem:estim-XNkk'} we have for any $\setv
\subseteq \setu$
\begin{equation*}
	\big| (X_{N;k,k'})^{|\setu \setminus \setv|}(J_{N;k,k'})^{|\setv|} \big|
  \leq \frac{1}{(12N)^{|\setu \setminus \setv |}\, 4^{|\setv |}} ,
\end{equation*}
and thus summing over $\displaystyle \mathfrak{v\subsetneq u}$ and
estimating $N^{-|\setu\setminus\setv|}$ by $N^{-1}$ we obtain
\[
  \displaystyle \sum_{\mathfrak{v\subsetneq u}} \big|
  (X_{N;k,k'})^{|\setu \setminus \setv|}
  (J_{N;k,k'})^{|\setv|} \big|
  \leq \frac{1}{N}\sum_{\mathfrak{v\subsetneq u}}\frac{1}{12^{|\setu \setminus \setv |}\, 4^{|\setv |}}.
\]
Further, from the binomial theorem we have
 $\sum_{\mathfrak{v\subsetneq u}}\frac{1}{12^{|\setu \setminus \setv |} 4^{|\setv |}}=\big(\frac{1}{12}+\frac14\big)^{|\setu|}-\frac1{4^{|\setu|}}
		=
		\frac1{3^{|\setu|}}
		-\frac1{4^{|\setu|}} $.
Using this, together with the case $\displaystyle \setv =\setu$, we obtain
\begin{equation*}
		\overline{e}^{2}_{\setu} (N)\leq
		\left(\frac{2}{3^{|\setu |}} - \frac{1}{4^{|\setu|}}\right)
		\frac{1}{N} +\frac{1}{N^{2}}\sum_{k=0}^{N-1}\sum\limits ^{N-1}_{\substack{
				k'=0\\
				k'\neq k}}
			(J_{N;k,k'})^{|\setu|}.
\end{equation*}%
Using again $|J_{N;k,k'}|\le 1/4$, we can separate out the contributions
for $k=0$ or $k'=0$, to obtain
\[
	\frac{1}{N^{2}}\sum_{k'=1}^{N-1}
	 |J_{N;0,k'} |^{|\setu|}
	\leq \frac{N-1}{4^{|\setu|} N^2} \leq \frac{1}{4^{|\setu|} N}
	\qquad\mbox{and}\qquad
	\frac{1}{N^{2}}\sum^{N-1}_{k=1} |J_{N;k,0} |\leq \frac{1}{4^{|\setu|}N}.
\]
Finally noting $(J_{N;k,k})^{|\setu|}\geq 0$ yields the desired result.
\end{proof}

\subsection{Estimates for $J_{N;k,k'}$}

In this subsection, we derive estimates for $J_{N;k,k'}$ for $k,k'\ge 1$.
In the following we will make use of the Fourier series for the real
$1$-periodic sawtooth function, defined on $[0,1)$ by
\[
 b(x):=\begin{cases}
 x-1/2 & \mbox{if } x\in(0,1),\\
 0 & \mbox{if } x=0,
\end{cases}
\]
and then extended to the whole of $\bbR$ by $b(x) = b(x+1)$ for all
$x\in\bbR$. Thus $b(x)$ is the periodic version of the first-degree
Bernoulli polynomial $B_1(x) = x-1/2$. It is well known (following, for
example, from the Dini criterion) that the symmetric partial sums in its
Fourier series converge to $b(x)$ pointwise for all $x\in \mathbb{R}$,
that is
\[
b(x) = \lim_{M\to\infty} \frac{\mathrm{i}}{2\pi}
\sum_{\substack{h=-M\\h\ne 0}}^M \frac{\exp(2\pi \mathrm{i} h x)}{h},\quad x \in \mathbb{R}.
\]
For notational simplicity we shall often omit the limit, writing simply
\[
b(x) = \frac{\mathrm{i}}{2\pi}
\sum_{h\ne 0} \frac{\exp(2\pi \mathrm{i} h x)}{h},\quad x \in \mathbb{R},
\]
but this is always to be understood as the limit of the symmetric partial
sum.

We have the following expression for $J_{N;k,k'}$, $k,k'\in \{1,\dotsc
,N-1\}$.
\begin{lemma}\label{lem:identity-J}
For $N$ prime and $k,k'\in \{1,\dotsc ,N-1\}$, the quantity $J_{N;k,k'}$
defined in \eqref{eq:def-JNkk'} satisfies
\begin{align}
		J_{N;k,k'}&
		=
		\frac{1}{4\pi^{2}}\frac{N}{N-1}
		\sum_{h \neq 0}
		\sum_{ \substack{
				h'\neq 0\\
				h'k'\equiv_N\, hk
		}}
		\frac{1}{hh'},\label{eq:J-nonzero}
\end{align}
where the double sum is to be in interpreted as the double limit
\begin{equation*}
		\sum_{h\neq 0}
		\sum_{\substack{
				h'\neq 0\\
				h'k' \equiv_N\, hk
		}}
		\frac{1}{hh'}
		:=
		\lim\limits_{M\to\infty}
		\lim\limits_{M'\to\infty}
		\sum_{h\in\{-M,\dots,M\}\setminus\{0\}}
		\sum_{\substack{
				h' \in\{-M' ,\dots,M' \}\setminus\{0\}\\
				h'k' \equiv_N\, hk
		}}
		\frac{1}{hh'}.
\end{equation*}
\end{lemma}

\begin{proof}
For $(x,y)\in(0,1)^2$ we have
		\begin{align*}
		B_{1}( x)B_{1}( y)
		&\,=
		\frac{1}{4\pi^2 }
		\sum_{h\not{=} 0}\sum_{h'\not{=} 0}
		\frac{e^{2\pi \mathrm{i}hx}}{h}
		\frac{e^{-2\pi \mathrm{i}h'y}}{h'} \\
		&:=
		\lim_{M\to\infty}
		\lim_{M'\to\infty}
		\frac{1}{4\pi^2 }
		\sum_{h \in\{-M ,\dots,M \}\setminus\{0\}}
		\sum_{h'\in\{-M',\dots,M'\}\setminus\{0\}}
		\frac{e^{2\pi \mathrm{i}hx}}{h}
		\frac{e^{-2\pi \mathrm{i}h'y}}{h'}.
\end{align*}
Thus for any $k,k'=1,\dots,N-1$ we have, noting that the finite sum over
$z$ may be interchanged with the implied limits,
\begin{align*}	
  J_{N;k,k'}
		& =\frac{1}{4\pi^{2}(N-1)}
		\sum_{h\not{=} 0}\sum_{h'\not{=} 0}
		\frac{1}{hh'} \sum^{N-1}_{z=1}
  \exp\bigg(2\pi \mathrm{i}\Big(\frac{hk-h'k'}{N}\Big) z\bigg) \\
		& =-\frac{1}{4\pi^{2}(N-1)}
		\sum_{h\not{=} 0}\sum_{h'\not{=} 0}
		\frac{1}{hh'}
		+
		\frac{1}{4\pi^{2}(N-1)}
		\sum_{h\not{=} 0}\sum_{h'\not{=} 0}
		\frac{1}{hh'} \sum ^{N-1}_{z=0}
  \exp\bigg(2\pi \mathrm{i}\Big(\frac{hk-h'k'}{N}\Big) z\bigg).
\end{align*}
The first term vanishes because it has as a factor the limit of the
product of symmetric partial sums of the odd function $1/h$. For the
second term we use
\begin{equation*}
		\sum ^{N-1}_{z=0} \exp(2\pi \mathrm{i} z(hk-h'k')/N) =
		\begin{cases}
		N & \text{if } hk-h'k'\equiv_N 0,\\
		0 & \text{if } hk-h'k'\not\equiv_N 0,
		\end{cases}
		%\label{eq:character2}
\end{equation*}
which leads to the desired formula.
\end{proof} 	

We now want to estimate $J_{N;k,k'}$ for $k,k'\ge 1$ using
\eqref{eq:J-nonzero}. It turns out that it suffices to consider
$J_{N;\kappa,1}$, for $\kappa=1,\dots,N-1$.
\begin{proposition}\label{prop:bd-e2-by-Jkappa1}
For $N$ prime, the quantity $\overline{e}^{2}_{\setu}(N)$ defined in
\eqref{eq:e2binom} satisfies
\begin{equation}
		\overline{e}^{2}_{\setu}(N)
		\le c_{\setu}
		\frac1N
		+
		\frac{1}{N}
		\sum_{\kappa=1}^{N-1}
		|J_{N;\kappa,1}|^{|\setu|}.
		\label{eq:bd-e2-by-Jkappa1}
\end{equation}
\end{proposition}
\begin{proof}
Because $N$ is prime, for each $k'\in \{1,\ldots,N-1\}$ there is a unique
inverse ${k'}^{-1} \in \{1,\dotsc ,N-1\}$ such that $k'{k'}^{-1} \equiv_N
1$, and therefore
\begin{equation*}
		h'k'\equiv_N hk
		\quad\Leftrightarrow \quad h'\equiv_N h ( k{k'}^{-1}).
\end{equation*}
It follows from \eqref{eq:J-nonzero} that
\begin{equation*}
		J_{N;k,k'} =J_{N; \kappa,1} ,
 \qquad\mbox{with}\qquad
	\kappa :=  k{k'}^{-1} \mod{N},
\end{equation*}
and since $\kappa$ runs over all of $\{1,\dots,N-1\}$ as $k'$ runs over
$\{1,\dots,N-1\}$, we have
\[
		\frac1{N^2}\sum_{k=1}^{N-1}\sum_{k'=1}^{N-1}
		(J_{N;k,k'})^{|\setu|}
		=
		\frac{N-1}{N^2}	\sum_{\kappa=1}^{N-1} (J_{N;\kappa,1})^{|\setu|}
		\le
		\frac{1}{N} \sum_{\kappa=1}^{N-1} |J_{N;\kappa,1}|^{|\setu|}.
\]
Applying this to Proposition~\ref{prop:bd-1st-term-ofe2} yields the
desired result.
\end{proof}

From Lemma~\ref{lem:identity-J} we have
\begin{align} \label{eq:JNkappa1}
	J_{N;\kappa,1}
    &= \frac{1}{4\pi^{2}}\frac{N}{N-1}
	\sum_{h\neq 0}
	\sum_{\substack{h'\neq 0\\ h'\equiv_N \, h\kappa}}
	\frac{1}{hh'}
    = \frac{1}{4\pi^{2}}\frac{N}{N-1}
    \lim_{M\to\infty} \lim_{M'\to\infty} S(M,M'),
\end{align}
where
\begin{align} \label{eq:S-def}
    S(M,M') :=
    	\sum_{h\in\{-M,\dots,M\}\setminus\{0\}}
    \sum_{\substack{
			h'\in\{-M',\dots,M'\}\setminus\{0\}\\
			h'\equiv_N h\kappa
	}}
	\frac{1}{hh'}.
\end{align}
To further simplify $J_{N;\kappa,1}$, we note that for $h,h'$ satisfying
$h'\equiv_N h\kappa$ with $\kappa \in \{1,\ldots,N-1\}$ we have
\begin{equation*}
	h\equiv_N 0 \ \ \ \Leftrightarrow \ \ \ h\kappa \equiv_N 0 \ \ \ \Leftrightarrow \ \ \ h'\equiv_N 0.
\end{equation*}
Hence, for the $h\equiv_N 0$ contribution to the double sum
\eqref{eq:S-def}  we have
\begin{equation*}
	\sum_{\substack{
			h\in\{-M,\dots,M\}\setminus\{0\}\\
			h\equiv_N 0
	}}
	\frac{1}{h}
	\sum_{\substack{
			h'\in\{-M',\dots,M'\}\setminus\{0\}\\
			h'\equiv_N 0
	}}\frac{1}{h'}
	=0.
\end{equation*}
Thus, we can restrict the double sum \eqref{eq:S-def} to $h\not\equiv_N 0$
so that
\begin{align}
	S(M,M') =
    \sum_{\substack{
			h\in\{-M,\dots,M\}\setminus\{0\}\\
			h\not\equiv_N 0
	}}
    \sum_{\substack{
			h'\in\{-M',\dots,M'\}\setminus\{0\}\\
			h'\equiv_N\, h\kappa
    }}
	\frac{1}{hh'}.
	\label{eq:S-simp}
\end{align}

We now assume $N\ge 3$ so that $N-1$ is even for $N$ prime. We can write
$h\not\equiv_N 0$ as
\begin{equation}
 h=\ell N+q,\quad\mbox{with}\quad \ell\in\mathbb{Z} \quad\mbox{and}\quad
 q\in \left\{-\tfrac{N-1}{2} ,...,\tfrac{N-1}{2}\right\} \setminus \{0\} =:R_N.
	\label{eq:h'-with-ell'}
\end{equation}
Then, we can write $h'\equiv_N h\kappa$ with $h\not\equiv_N 0$ as
\[
  h'=\ell' N+ r(q\kappa,N), \quad\mbox{with}\quad \ell'\in\mathbb{Z},
\]
where $r(j,N)$ is the unique integer congruent to $j\bmod N$ with the
smallest magnitude. More precisely, the function
$r(\cdot,N)\colon\mathbb{Z} \to R_N \cup \{0\}$ is defined for $j\ge 0$ by
\begin{align} \label{eq:r-def}
	r(j,N):=
	\begin{cases}
	j \bmod N
	& \text{ if}\quad j\bmod N \le \frac{N-1}{2},
	\\
	j \bmod N - N
	& \text{ if}\quad j\bmod N > \frac{N-1}{2},
	\end{cases}
\end{align}
and extended to all integers $j$ by $r(j,N) =r(j+N,N)$. It follows that
for $j>0$ we have $r(-j,N) = r(N-j\bmod N,N) = -r(j,N)$. Hence the
function is both $N$-periodic and odd. If $N$ divides $j$, then we have
$r(j,N)= 0$, but otherwise $r(j,N) \in R_N$.

Using these representations of $h$ and $h'$, the double limit in
$J_{N;\kappa,1}$ as in \eqref{eq:JNkappa1} can be rewritten as follows.

\begin{lemma}\label{lem:J1}
For $N\ge 3$ prime and $\kappa\in\{1,\dots,N-1\}$, the quantity
$J_{N;\kappa,1}$ given by \eqref{eq:JNkappa1} satisfies
\begin{align} \label{eq:J1}
 J_{N;\kappa,1} =
 \frac{1}{2\pi^{2}}\frac{N}{N-1} \sum_{q=1}^{(N-1)/2}
 \Bigg(\frac{1}{q} - \sum_{\ell=1}^\infty \frac{2q}{(\ell N)^2-q^2 }\Bigg)
 \Bigg(\frac{1}{r(q\kappa,N)} - \sum_{\ell'=1}^\infty \frac{2\,r(q\kappa,N)}{(\ell' N)^2-r(q\kappa,N)^2 }\Bigg),
\end{align}
where $r(\cdot,N)$ is defined as in \eqref{eq:r-def}.
\end{lemma}

\begin{proof}
We begin with the expression \eqref{eq:JNkappa1} for $J_{N;\kappa,1}$.
Writing $M =L N + Q$ and $M'=L'N+Q'$ with $L,L'\in\bbN$ and $Q,Q'\in
R_N\cup\{0\}$, the double sum \eqref{eq:S-simp} can be rewritten as
\begin{align} \label{eq:S-prod}
 S(M,M')
 &= \sum_{\substack{\ell\in\bbZ,\, q\in R_N\\|\ell N+q|\le LN+Q}}\;
   \sum_{\substack{\ell'\in\bbZ,\, q'\in R_N\\ |\ell' N+q'|\le L'N+Q'\\ q'\equiv_N\, q\kappa}}
   \frac{1}{\ell N+q}\; \frac{1}{\ell' N + q'} \nonumber\\
 &= \sum_{\satop{q,q'\in R_N}{q'\equiv_N\, q\kappa}}
 \Bigg(\sum_{\satop{\ell=-L}{|\ell N+q| \le LN+Q}}^L \frac{1}{\ell N+q }\Bigg)
 \Bigg(\sum_{\satop{\ell'=-L'}{|\ell' N+q'| \le L'N+Q'}}^{L'} \frac{1}{\ell' N + q'}\Bigg),
\end{align}
where we used the fact that the inequalities in the summation conditions
cannot hold if $|\ell|> L$ or $|\ell'|> L'$.

First we consider the sum over $\ell$ in \eqref{eq:S-prod}. Since the
condition $|\ell N+q|\le LN + Q$ always holds for $|\ell|\le L-1$, we can
write
\[
  \sum_{\satop{\ell=-L}{|\ell N+q| \le LN+Q}}^L \frac{1}{\ell N+q }
  = \sum_{\ell=-L}^L \frac{1}{\ell N+q } - \sum_{\satop{\ell=\pm L}{|\ell N+q| > LN+Q}} \frac{1}{\ell N+q},
\]
where we have
\[
  \sum_{\ell=-L}^L \frac{1}{\ell N+q }
  = \frac{1}{q} + \sum_{\ell=1}^L \left(\frac{1}{\ell N+q } + \frac{1}{-\ell N+q }\right)
  = \frac{1}{q} - \sum_{\ell=1}^L \frac{2q}{(\ell N)^2-q^2 }
\]
and
\[
  \Bigg|\sum_{\satop{\ell=\pm L}{|\ell N+q| > LN+Q}} \frac{1}{\ell N+q} \Bigg|
  \le \frac{2}{LN+Q}
  \le \frac{2}{LN - N/2}
  \to 0 \qquad\mbox{as}\qquad L\to\infty.
\]
Thus we conclude that
\[
  \lim_{L\to\infty} \sum_{\satop{\ell=-L}{|\ell N+q| \le LN+Q}}^L \frac{1}{\ell N+q }
  = \frac{1}{q} - \lim_{L\to\infty} \sum_{\ell=1}^L \frac{2q}{(\ell N)^2-q^2 }
  = \frac{1}{q} - \sum_{\ell=1}^\infty \frac{2q}{(\ell N)^2-q^2 }
  =: P_N(q).
\]
The sum over $\ell'$ in \eqref{eq:S-prod} is similar.

Now since the double limit of $S(M,M')$ exists as $M\to\infty$ and
$M'\to\infty$, it must equal the double limit of the last expression in
\eqref{eq:S-prod} as $L\to\infty$ and $L'\to\infty$, with arbitrary $Q$
and~$Q'$. (This is because for a particular pair $(Q,Q')$, the last
expression in \eqref{eq:S-prod}, when interpreted as a sequence in the
double index $(L,L')$, can be considered as a subsequence of the
convergent sequence $S(M,M')$ with double index $(M,M')$.) Hence we obtain
\[
  \lim_{M\to\infty} \lim_{M'\to\infty} S(M,M')
  = \sum_{\satop{q,q'\in R_N}{q'\equiv_N\, q\kappa}} P_N(q)\, P_N(q')
  = \sum_{q\in R_N} P_N(q)\, P_N(r(q\kappa,N)),
\]
where we used the fact that for a given $q\in R_N$, the only value
of $q'\in R_N$ that satisfies $q'\equiv_N q\kappa$ is $q' = r(q\kappa,N)$.

Finally, we observe that $P_N(-q) = - P_N(q)$, and $P_N(r(-q\kappa,N)) =
-P_N(r(q\kappa,N))$ since $r(-q\kappa,N) = -r(q\kappa,N)$. Thus the
contributions of $q$ and $-q$ to the sum are the same, and so we only need
to sum over the positive values of $q$ and then double the result.
Applying the result in \eqref{eq:JNkappa1} completes the proof.
\end{proof}		

Now we estimate the magnitude of $J_{N;\kappa,1}$.

\begin{lemma}\label{lem:J1-bound}
For $N\ge 3$ prime and $\kappa\in\{1,\dots,N-1\}$, the quantity
$J_{N;\kappa,1}$ from \eqref{eq:J1} satisfies
\begin{equation*}
  |J_{N;\kappa,1}| \le \frac{1}{2\pi^2} \frac{N}{N-1}
  \left(T_N(\kappa) + \frac{10\pi^2 \ln N}{9N}\right),
\end{equation*}
where
\begin{equation} \label{eq:T-def}
  T_N(\kappa) := \sum_{q=1}^{(N-1)/2} \frac{1}{q\,|r(q\kappa,N)|} < \frac{\pi^2}{6}.
\end{equation}
\end{lemma}
\begin{proof}
We expand the two factors in the sum over $q$ in \eqref{eq:J1} and then
apply the triangle inequality to obtain
\[
  |J_{N;\kappa,1}| \le \frac{1}{2\pi^2} \frac{N}{N-1} \big(T_N(\kappa) + A_1 + A_2 + A_3\big),
\]
with
\begin{align*}
  A_1 &:= \sum_{q=1}^{(N-1)/2} \frac{1}{|r(q\kappa,N)|} \Bigg(\sum_{\ell=1}^\infty \frac{2q}{(\ell N)^2-q^2 }\Bigg), \\
  A_2 &:= \sum_{q=1}^{(N-1)/2} \frac{1}{q} \Bigg(\sum_{\ell'=1}^\infty \frac{2\,|r(q\kappa,N)|}{(\ell' N)^2-r(q\kappa,N)^2 }\Bigg), \\
  A_3 &:= \sum_{q=1}^{(N-1)/2} \Bigg(\sum_{\ell=1}^\infty \frac{2q}{(\ell N)^2-q^2 }\Bigg)
  \Bigg(\sum_{\ell'=1}^\infty \frac{2\,|r(q\kappa,N)|}{(\ell' N)^2-r(q\kappa,N)^2 }\Bigg).
\end{align*}
Since $q\le N/2\le \ell N/2$ and $|r(q\kappa,N)|\le N/2\le\ell' N/2$, we
have
\[
  \sum_{\ell=1}^\infty \frac{2q}{(\ell N)^2-q^2 }
  \le \sum_{\ell=1}^\infty \frac{N}{(\ell N)^2-(\ell N/2)^2 }
  = \frac{4}{3N} \sum_{\ell=1}^\infty \frac{1}{\ell^2}
  = \frac{2\pi^2}{9N},
\]
and
\[
  \sum_{\ell'=1}^\infty \frac{2\,|r(q\kappa,N)|}{(\ell' N)^2-r(q\kappa,N)^2}
  \le \sum_{\ell'=1}^\infty \frac{N}{(\ell' N)^2-(\ell' N/2)^2 }
  %= \frac{4}{3N} \sum_{\ell'=1}^\infty \frac{1}{{\ell'}^2}
  = \frac{2\pi^2}{9N}.
\]
Moreover, we have
\[
  \sum_{q=1}^{(N-1)/2} \frac{1}{q} \le 1 + \int_1^{(N-1)/2} \frac{1}{t}\,\rd t \le 2\ln N
  \quad\mbox{and}\quad
  \sum_{q=1}^{(N-1)/2} \frac{1}{|r(q\kappa,N)|} = \sum_{t=1}^{(N-1)/2} \frac{1}{t} \le 2\ln N,
\]
where in the prenultimate step we used the fact that $|r(q\kappa,N)|$
takes all the values from $1$ to $(N-1)/2$ exactly once as $q$ runs from
$1$ to $(N-1)/2$. These estimates lead to
\[
  A_1 + A_2 + A_3
  \le \frac{4\pi^2\,\ln N}{9N} + \frac{4\pi^2\,\ln N}{9N} + \frac{N-1}{2}\frac{4\pi^4}{81N^2}
  \le \frac{\pi^2\,\ln N}{9N} \left(8 + \frac{2\pi^2}{9\ln 3}\right)
  \le \frac{10\pi^2\,\ln N}{9N}.
\]
On the other hand, a crude estimate for $T_N(\kappa)$ follows from the
Cauchy-Schwarz inequality:
\begin{align*}
		T_N(\kappa) &
		\leq \Bigg(\sum ^{(N-1) /2}_{q=1}\frac{1}{q^{2}}\Bigg)^{1/2}
        \Bigg(\sum ^{(N-1) /2}_{q=1}\frac{1}{r(q\kappa,N)^{2}}\Bigg)^{1/2}
		=\Bigg(\sum ^{(N-1)/2}_{q=1}\frac{1}{q^{2}}\Bigg)^{1/2}
       \Bigg(\sum ^{(N-1) /2}_{t=1}\frac{1}{t^{2}}\Bigg)^{1/2}
        < \frac{\pi^{2}}{6}.
\end{align*}
This completes the proof.
\end{proof}

Numerical experiments show that the value of $T_N(\kappa)$ is much smaller
than the crude bound $\pi^2/6$ for most values of $\kappa$, and have led
us to the following conjecture. Note that we have $r(q(N-\kappa),N) =
r(-q\kappa,N) = -r(q\kappa,N)$, and so $T_N(N-\kappa) = T_N(\kappa)$.
Moreover, from \eqref{eq:J1} we conclude that
\[
  J_{N;N-\kappa,1} = - J_{N;\kappa,1}.
\]
Since we are only interested in the magnitude of $J_{N;\kappa,1}$ (see
Proposition~\ref{prop:bd-e2-by-Jkappa1}), it suffices to consider only
$\kappa \in R_N^+ := \{1,2,\ldots,(N-1)/2\}$.

\begin{conjecture} \label{conj:alpha}
For $N\ge 3$ prime and $\kappa\in R_N^+$, with $T_N(\kappa)$ as defined in
\eqref{eq:T-def}, let $( \kappa_{j})$ for $j\in R^{+}_{N}$ be an ordering
of the elements of $R_N^+$ such that $(T_{N}(\kappa _{j}))$ is
non-increasing. The conjecture is that there exist $C_{1}, C_{2}>0$ and
$\alpha\ge 2$ independent of $N$ such that
\begin{equation}
		T_{N}(\kappa _{j}) \le C_{1}\frac{(\ln N)^{\alpha}}{N} \ \
\text{ for all } \ \ \ j >C_{2}\,(\ln N)^{\alpha}.\label{eq:conj}
\end{equation}
\end{conjecture}

{Conjecture~\ref{conj:alpha}} together with Lemma~\ref{lem:J1-bound}
lead to an estimate for $|J_{N;\kappa_j,1}|$ of the following form:
\[
 |J_{N;\kappa_j,1}|
 \le
 \begin{cases}
 C_3  & \text{ for } j\le C_2\, (\ln N)^\alpha, \\
 C_4 \displaystyle\frac{(\ln N)^\alpha }{N} &\text{ for } j>C_2\, (\ln N)^\alpha,\\
 \end{cases}
\]
where $C_3$ and $C_4$ are known numerical constants. We will use this
bound in the next subsection to obtain the desired result for the mean of
the worst-case error.

\subsection{Final results}

Now we are ready to state our main results.

\begin{theorem}\label{thm:bd-bar-e-u-N}
Suppose that {Conjecture~\ref{conj:alpha}} holds with some
$\alpha\ge2$. For arbitrary $\setu\subseteq\{1:s\}$ and any prime number
$N\ge 3$ such that ${(\ln N)^{\alpha}}/{N}\le 1$, the quantity
$\overline{e}_{\setu}(N)$ defined in \eqref{eq:def-bar-e-u} satisfies
\begin{equation}
		\overline{e}_{\setu}(N)
		\le C_{\setu} \frac{(\ln N)^{\alpha/2}}{\sqrt{N}},
\end{equation}
where
\[
		C_{\setu}:=
		\sqrt{ c_{\setu}
			+ 2C_{2}\Big(\frac{23}{24}\Big)^{|\setu |}
			+\Big(
			\frac{3C_{1}}{4\pi^{2}} +\frac{5}{6}\Big)^{|\setu |}
        }.
\]
Here, the constant $c_{\setu}$ is as in
Proposition~\ref{prop:bd-1st-term-ofe2}, and $C_1,C_2$ are as in
{Conjecture~\ref{conj:alpha}}.
\end{theorem}

\begin{proof}
From Proposition~\ref{prop:bd-e2-by-Jkappa1} together with
$J_{N;N-\kappa,1} = - J_{N;\kappa,1}$, we have
\begin{equation}
		\overline{e}^{2}_{\setu}(N)
		\le c_{\setu}
		\frac1N
		+
		\frac{2}{N}
		\sum_{j=1}^{(N-1)/2}
		|J_{N;\kappa_j,1}|^{|\setu|}.
		\label{eq:bd-e2-by-Jkappa1}
\end{equation}
For $j\leq C_{2}(\ln N)^{\alpha}$, we use $T_N(\kappa_j)\le \pi^2/6$, $\ln
N/N\le 1$ and $N/(N-1)\le 3/2$ in Lemma~\ref{lem:J1-bound} to obtain
\begin{equation*}
 |J_{N;\kappa_{j},1}|
 \le \frac{1}{2\pi^{2}}\frac{N}{N-1}\bigg( \frac{\pi^2}{6} +\frac{10\pi^{2}}{9}\frac{\ln N}{N}\bigg)
 \le \frac{1}{2\pi^{2}}\frac{3}{2}\bigg( \frac{\pi^2}{6} +\frac{10\pi^{2}}{9}\bigg)
 = \frac{23}{24}.
\end{equation*}
For $j > C_{2}(\ln N)^{\alpha }$, we use $\ln N\geq 1$, $N/(N-1)\le 3/2$
and {Conjecture~\ref{conj:alpha}} to obtain
\begin{align*}
 |J_{N;\kappa_{j},1}|
 \le \frac{1}{2\pi^{2}}\frac{N}{N-1}\bigg( C_{1}\frac{(\ln N)^{\alpha }}{N} +\frac{10\pi^{2}}{9}\frac{\ln N}{N}\bigg)
 \le \bigg( \frac{3C_1}{4\pi^2} +\frac{5}{6}\bigg)\frac{(\ln N)^{\alpha}}{N}.
\end{align*}
Combining these and using $(\ln N)^{\alpha }/{N} \leq 1$, we obtain
\begin{align*}
 \sum_{j=1}^{(N-1)/2} |J_{N;\kappa_j,1} |^{|\setu|}
 &\le \sum_{1\leq j\leq C_{2}(\ln N)^{\alpha }} \bigg(\frac{23}{24}\bigg)^{|\setu|}
 + \sum_{C_{2}(\ln N)^{\alpha } < j\leq (N-1) /2} \bigg( \frac{3C_1}{4\pi^2} +\frac{5}{6}\bigg)^{|\setu|}
   \bigg(\frac{(\ln N)^{\alpha}}{N}\bigg)^{|\setu|}
 \\
 &\le C_{2}(\ln N)^{\alpha }\left(\frac{23}{24}\right)^{|\setu |}
 + \frac{N-1}{2}\bigg( \frac{3C_1}{4\pi^2} +\frac{5}{6}\bigg)^{|\setu|} \frac{(\ln N)^{\alpha}}{N}\\
 &\le \bigg( C_{2}\left(\frac{23}{24}\right)^{|\setu |} +
		\frac{1}{2} \left(\frac{3C_{1}}{4\pi^{2}} + \frac{5}{6}\right)^{|\setu |}
	\bigg) (\ln N)^{\alpha }.
\end{align*}
This together with \eqref{eq:bd-e2-by-Jkappa1} yields the required
result.
\end{proof}

\begin{corollary}\label{cor:final-cor}
Suppose that {Conjecture~\ref{conj:alpha}} holds with some $\alpha\ge
2$. Let $N\ge 3$ be a prime number. Suppose that the weights $\bsgamma =
(\gamma_\setu)_{\setu}$ satisfy
\[
  C := \sum_{|\setu|<\infty} \gamma_\setu\, C_\setu < \infty,
\]
where $C_{\setu}$ is the constant as in Theorem~\ref{thm:bd-bar-e-u-N}.
Then, the generating-vector-averaged worst-case error $\overline{e}^{2}(N)
$ defined as in \eqref{eq:def-bar-e2N} satisfies
\begin{equation}
		\overline{e}^{2}(N)\le C \frac{(\ln N)^{\alpha}}{{N}},
\end{equation}
with $C>0$ independent of $s$ and $N$.
As a consequence, there exists a generating vector $\bsz^*\in
Z_{N}^s=\{z\in \mathbb{Z}\mid 1\le z\le N-1 \}^s$ that attains the
worst-case error
\begin{equation}
		{e}(N,\bsz^*)\le \sqrt{C} \frac{(\ln N)^{\alpha/2}}{\sqrt{N}}.
		\end{equation}
\end{corollary}
	
\begin{proof}
From \eqref{eq:def-bar-e2N} and Theorem~\ref{thm:bd-bar-e-u-N}, we have
\begin{equation}
		\overline{e}^{2}(N)
		\le \sum_{\emptyset \neq \setu \subseteq \{1:s\}} \gamma_{\setu} C_{\setu} \frac{(\ln N)^{\alpha}}{{N}}
  \le C \frac{(\ln N)^{\alpha}}{{N}}.
\end{equation}
Now, recall that $\overline{e}^{2}(N)$ is defined in
\eqref{eq:def-bar-e2N} as the average of ${e}^2(N,\bsz)$ over all possible
$\bsz$. Thus, there must be at least one $\bsz^*$ such that
\[
		{e}^2(N,\bsz^*)\le C \frac{(\ln N)^{\alpha}}{{N}},
\]
which yields the second statement.
\end{proof}

\section{Numerical experiments on the conjecture} \label{sec:num}

In this section, we present numerical evidence relating to
{Conjecture~\ref{conj:alpha}}. We compute the numbers
$\{T_N(\kappa)\}_{\kappa=1}^{(N-1)/2}$, given by \eqref{eq:T-def} for
varying $N$. For each fixed $N$, we sort these values in non-increasing
order, which we write as $(T_N(\kappa_j))_{j=1,\dots,(N-1)/2}$, plot the
values, and make a guess of the constants $C_1$, $C_2$ in
{Conjecture~\ref{conj:alpha}}. We used Julia 0.6.2.\ for the
experiments below.

Figure~\ref{fig:50kto100k} shows the values of $\frac{N}{(\ln N)^\alpha
}T_N(\kappa_j)$ against $j/(\ln N)^\alpha$ for $j=1,\dots,(N-1)/2$ with
$\alpha=2,3$, and $N=50021,74687,99991$. We see that for both $\alpha=2$
and $3$ and these values of $N$ we can take constants $C_1$, $C_2$ such
that for all $j/(\ln N)^\alpha>C_2$ with $j=1,\dots,(N-1)/2$ we have
$T_N(\kappa_j){N}/{(\ln N)^\alpha }\le C_1$: for example, $C_1=20$ and
$C_2=10$. This is consistent with {Conjecture~\ref{conj:alpha}},
especially for $\alpha = 3$. Of course, we cannot be certain even in this
case that the bounds will hold for very large $N$, with these or any
constants. But even if the conjecture fails, the numerical experiments
give us confidence, even for $\alpha=2$, that the bounds in
Theorem~\ref{thm:bd-bar-e-u-N} will hold with $C_1=20$ and $C_2=10$ for
$N$ up to at least a few hundred thousand.

\begin{figure}%[H]%requres \usepackage{subfig} for \subfloat and {here} for [H] option denoting here
	\centering
	\subfloat[]{
		\includegraphics[width=0.9\textwidth]{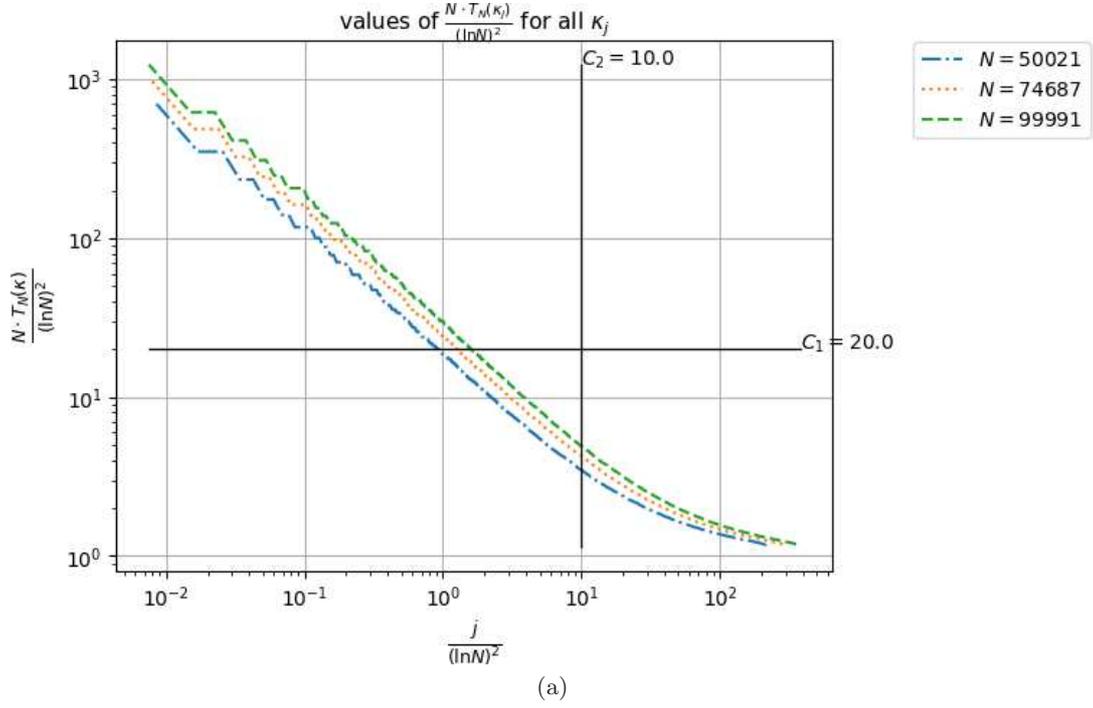}
	}
	
	%\subfloat[]{
	%   \includegraphics[width=0.49\textwidth]{alpha2pt5_50kto100k.png}
	%}
	\subfloat[]{
		\includegraphics[width=0.9\textwidth]{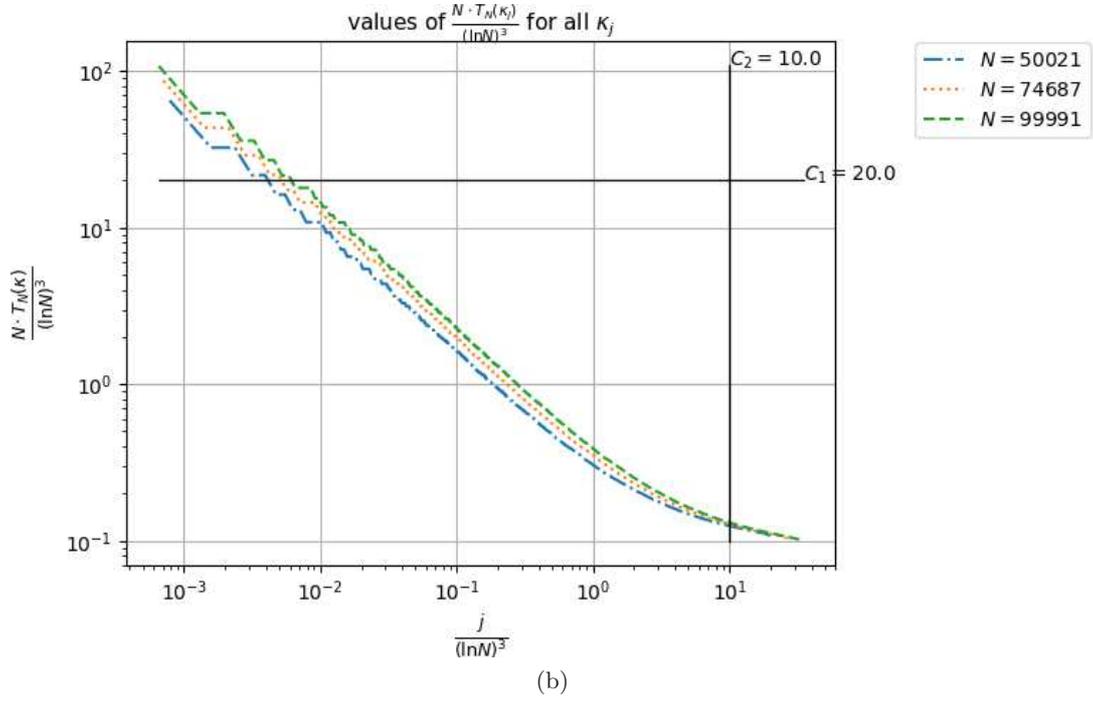}
	}
	\caption{Values of $T_N(\kappa_j){N}/{(\ln N)^\alpha }$, against $j/(\ln N)^\alpha$ for $j=1,\dots,(N-1)/2$. Top: $\alpha=2$. Bottom: $\alpha=3$. We see there exist constants $C_1$, $C_2$ such that for all $j/(\ln N)^\alpha>C_2$ we have $T_N(\kappa_j){N}/{(\ln N)^\alpha }\le C_1$: for example, $C_1=20$ and $C_2=10$.}
	\label{fig:50kto100k}
\end{figure}

\section{Conclusion}\label{sec:conclusion}

In this paper, we considered the worst-case error for unshifted lattice
rules without randomisation. A conjecture to support the error estimate
was proposed. Given the conjecture, we showed the existence of a
generating vector that attains the worst-case error $1/\sqrt{N}$, up to a
logarithmic factor. Numerical experiments suggest that the conjecture is
plausible.

{
\section*{Acknowledgements}
We gratefully acknowledge the financial support from the Australian
Research Council (FT130100655 and DP180101356).}

\bibliographystyle{plain}
%\bibliography{/Users/kazashiy/Dropbox/aResearch/mybib}

\end{document}